\documentclass[a4paper,12pt]{amsart}

\usepackage{hyperref}

\usepackage {amsfonts, amssymb, amscd,amsthm, amsmath, eucal}

\usepackage{amsbsy}
\usepackage[all]{xy}

\theoremstyle{plain}
{
  \newtheorem{thm}{Theorem}[section]
  
  \newtheorem{cor}[thm]{Corollary}
  \newtheorem{lem}[thm]{Lemma}

  \newtheorem{sssection}[thm]{}
}

{

}
\renewcommand{\subsubsection}{\sssection\rm}

\newcommand{\can}{\text{\rm can}}

\newcommand{\id}{\text{\rm id}}
\newcommand{\pr}{\text{\rm pr}}

\newcommand{\inc}{\text{\rm inc}}

\newcommand{\const}{\text{\rm const}}
\newcommand{\Spec}{\text{\rm Spec}\,}

\newcommand{\Aff}{\mathbb {A}}
\newcommand{\ZZ}{\mathbb {Z}}

\newcommand \hra {\hookrightarrow }

\newcommand{\ttf}{{\text{f}}}

\renewcommand \id{\operatorname{id}}

\renewcommand \phi\varphi

\newcommand{\ad}{ad}

\newcommand{\Asc}{\mathcal A}
\newcommand{\Ysc}{Y}
\newcommand{\Zsc}{\Asc}
\newcommand{\QQ}{{\mathbb Q}}

\DeclareMathOperator{\et}{\text{\it \'et}}
\DeclareMathOperator{\GL}{GL}
\DeclareMathOperator{\PGL}{PGL}
\DeclareMathOperator{\End}{End}

\newtheorem*{thm*}{Theorem}

\begin{document}

\title
{Principal bundles of reductive groups over affine schemes}

\author{I.~Panin}
\thanks{The first author acknowledges support of the
RFBR projects 09-01-91333-NNIO-a and 10-01-00551}
\author{A.~Stavrova}\thanks{The second author acknowledges support of DFG SFB/TR 45; RFBR~10-01-00551;
research program 6.38.74.2011 of St. Petersburg State University.}


\maketitle

\begin{abstract}
Let $R$ be a semi-local regular domain containing an infinite
perfect field k, and let $K$ be the field of fractions of R. Let $G$ be a
reductive semi-simple simply connected $R$-group scheme such that
each of its $R$-indecomposable factors is isotropic.
We prove that for any Noetherian affine scheme $\Asc=\Spec A$ 
over $k$, the kernel of the map
$$ H^1_{\text{\'et}}(\Asc\times_{\Spec k} {\Spec R},G)\to H^1_{\text{\'et}}(\Asc\times_{\Spec k} \Spec K,G) $$
\noindent
induced by the inclusion of $R$ into $K$, is trivial.
If $R$ is the semi-local ring of several points on
a $k$-smooth scheme, then it suffices to require that $k$ is infinite
and keep the same assumption concerning $G$.
The results extend the Serre---Grothendieck
conjecture for such $R$ and $G$, proved in~\cite{PaSV}.

\end{abstract}

\section{Introduction}
\label{Introduction}

Recall that an $R$-group scheme $G$ is called {\it reductive}
(respectively, {\it semi-simple}; respectively, {\it simple}), if it is affine and smooth
as an $R$-scheme and if, moreover, for each ring homomorphism
$s:R\to\Omega(s)$ to an algebraically closed field $\Omega(s)$,
its scalar extension $G_{\Omega(s)}$
is a reductive (respectively, semi-simple; respectively, simple) algebraic
group over $\Omega(s)$.
This notion of a
simple $R$-group scheme coincides with the notion of a {simple
semi-simple $R$-group scheme of
\cite[Exp.~XIX, Def.~2.7 and Exp.~XXIV, 5.3]{SGA3}.}

Such an $R$-group scheme $G$ is called {\it simply-connected\/}
(respectively, {\it adjoint\/}), if for any homomorphism
$s:R\to \Omega(s)$ of $R$ to an algebraically closed field
$\Omega(s)$, the group $G_{\Omega(s)}$ is a
simply-connected (respectively, adjoint) $\Omega(s)$-group scheme (see
~\cite[Exp.~XXII, Def.~4.3.3]{SGA3}). A simple group scheme $G$ is called {\it isotropic}, if
it contains a split torus $G_{m,R}$.

We prove the following theorem, which is an extension
of the results on the Serre---Grothendieck
conjecture obtained in~\cite{PaSV}.

\begin{thm}\label{th:main}
Let $R$ be a regular semi-local domain containing a infinite perfect
field $k$. Let $K$ be the field of fractions
of $R$.
Let $G$ be an isotropic simple simply connected $R$-group scheme.

For any  Noetherian affine scheme $\Asc=\Spec A$ 
over $k$, the map
$$
H^1_{\text{\'et}}(\Asc\times_{\Spec k} \Spec R,G)\to H^1_{\text{\'et}}(\Asc\times_{\Spec k} \Spec K,G)
$$
induced by the inclusion of $R$ into $K$, has trivial kernel.
\end{thm}

This theorem is deduced, via a theorem of D. Popescu, from its following ``geometric'' version.

\begin{thm}\label{th:main-geometric}
Let $R$ be a semi-local ring of several points on a $k$-smooth scheme over an infinite field $k$. Let $K$ be the field of fractions
of $R$.
Let $G$ be an isotropic simple simply connected $R$-group scheme.

For any  Noetherian affine scheme $\Asc=\Spec A$ 
over $k$, the map
$$
H^1_{\text{\'et}}(\Asc\times_{\Spec k} \Spec R,G)\to H^1_{\text{\'et}}(\Asc\times_{\Spec k} \Spec K,G)
$$
induced by the inclusion of $R$ into $K$, has trivial kernel.
\end{thm}

The proof of this theorem is given in Section~\ref{sec:mainproof}. It uses, on one hand, the constructions
and results of~\cite{PaSV}. On the other hand, it requires the following extension of Quillen's local-global principle
for projective modules~\cite[Theorem 1]{Q}, due to L.-F. Moser
(previously announced without proof by Raghunathan~\cite[Theorem 2]{R0}, and hinted in~\cite{BCW}).

\begin{thm*}\label{th:Mo}\cite[Korollar 3.5.2]{Mo}
Let $A$ be a Noetherian commutative ring, $G$ a group scheme over $A$ admitting a closed
embedding $G\to \GL_{n,A}$ for some
$n\ge 1$. Let $E$ be a $G$-torsor over $\Aff^1_A$, such that $E$ is trivial on $\Aff^1_{U_i}$ for
all elements $U_i$ of a Zariski covering $\Spec A=\bigcup U_i$, and on the zero-section $\{0\}\times \Spec A$. Then $E$ is trivial.
\end{thm*}

Using this local-global principle one more time, we obtain the following corollary of Theorem~\ref{th:main}.

\begin{cor}\label{cor:nonlocal-line}
Let $S$ be a Noetherian ring such that for any maximal ideal $m$ of $S$, the local ring $S_m$ satisfies
the conditions imposed on $R$ in Theorem~\ref{th:main}, or in Theorem~\ref{th:main-geometric}. Let $G$ be a
simple simply connected $S$-group scheme admitting a closed embedding $G\to\GL_{n,S}$
for some $n\ge 1$, and such that for any maximal ideal $m$ of $S$, the group $G_{S_m}$ is isotropic.
Let $K$ be the field of fractions of $S$. Then the natural map
$$
H^1_{\et}(S[t],G)\to H^1_{\et}(K(t),G)
$$
has trivial kernel.
\end{cor}
\begin{proof}
Consider the composition
$$
H^1_{\et}(S[t],G)\to H^1_{\et}(K[t],G)\to H^1_{\et}(K(t),G).
$$
By~\cite[Prop. 2.2]{C-TO} the map $H^1_{\et}(K[t],G)\to H^1_{\et}(K(t),G)$ has trivial kernel.
It remains to prove that $H^1_{\et}(S[t],G)\to H^1_{\et}(K[t],G)$ has trivial kernel.
By the local-global principle, we can substitute $S$ by its localization at a maximal ideal, and then
apply Theorem~\ref{th:main} for $\Asc=\Aff^1_k$.
\end{proof}

{\bf Remark 1.} The conditions of Corollary~\ref{cor:nonlocal-line} on $S$ are satisfied, in particular, if
$S$ is a (not necessarily semilocal) regular ring containing an infinite perfect field, or
if $\Spec S$ is a smooth affine scheme over an infinite field.

\begin{cor}\label{cor:perfect-line}
Let $S$, $G$ be as in Corollary~\ref{cor:nonlocal-line}. Assume moreover that the field of fractions $K$
of $S$ is perfect. Then the map
$$
H^1_{\et}(S[t],G)\to H^1_{\et}(S,G)
$$
induced by evaluation at $t=0$, has trivial kernel.
\end{cor}
\begin{proof}
We have a commutative diagram
\begin{equation}
\xymatrix{
H^1_{\et}(S[t],G)\ar[rr]^{t=0}\ar[d]_{} &&
H^1_{\et}(S,G)\ar[d]^{}  &\\
H^1_{\et}(K[t],G)\ar[rr]^{t=0}&& H^1_{\et}(K,G).
}
\end{equation}
Since $K$ is perfect, the bottom line is an isomorphism by the main result of~\cite{RR}. The left vertical
line has trivial kernel by~Corollary~\ref{cor:nonlocal-line}.
\end{proof}

{\bf Remark 2.} The conditions of Corollary~\ref{cor:perfect-line} on $S$
are satisfied, in particular, if $S$ is a (not necessarily semilocal) regular ring containing $\QQ$, or
if $\Spec S$ is a smooth affine scheme over a field of characteristic 0.

{\bf Remark 3.} All the above results
can be easily extended to the case where $G$ is not simple but semisimple, and satisfies the following
isotropy condition: every semisimple normal subgroup of $G$ is isotropic. This follows
from Faddeev---Shapiro lemma~\cite[Exp. XXIV Prop. 8.4]{SGA3} (see also~\cite[Section 12]{PaSV}).


\section{Construction of a bundle over an affine line}\label{sec:constr}

Let $k$, $R$, $K$, $G$ be as in Theorem~\ref{th:main-geometric}.
Let $\Zsc$ be any scheme over $k$. In this section we show that any (\'etale) principal
$G$-bundle over $\Zsc\times_{\Spec k}\Spec R$ which becomes
trivial over $\Zsc\times_{\Spec k}\Spec K$ can be substituted by a principal $G$-bundle $P_t$ over $\Aff^1_R\times\Zsc$ which is
trivial over $(\Aff^1_R)_f\times \Zsc$,
for some monic polynomial $f\in R[t]$, in such a way that the triviality of this new bundle implies the triviality
of $P$. The argument is an extension of the argument of~\cite[\S 6]{PaSV}.

For compatibility with~\cite{PaSV}, in this section we denote $R$ by $\mathcal O$. We set
$$
\Ysc:=\Zsc\times_{\Spec k}\Spec \mathcal O=\Zsc\times_{\Spec k}\Spec R
$$
for shortness.

Fix a smooth affine $k$-scheme $X$ and a finite family of points
$x_1,x_2,\dots,x_n$ on $X$, such that $\mathcal O=\mathcal
O_{X,\{x_1,x_2,\dots,x_n\}}$. Set $U:= \text{Spec}(\mathcal O)$. Let
$$
\can:U\to X
$$
be the canonical map.
Further, consider a simple simply-connected $U$-group scheme $G$.

Let $P$ be a principal $G$-bundle over the scheme $\Ysc$ which is trivial
over $\Ysc\times_{\Spec \mathcal O}\Spec K$. We may and will
assume that for certain $\text{f} \in \mathcal O$ the principal
$G$-bundle $P$ is trivial over $\Ysc\times_{\Spec \mathcal O}\Spec \mathcal O_{\text{f}}$. Shrinking
$X$ if necessary, we may secure the following properties
\par\smallskip
(i) The points $x_1,x_2,\dots,x_n$ are still in $X$.
\par\smallskip
(ii) The group scheme $G$ is defined over $X$ and it is a simple
group scheme. We will often denote this $X$-group scheme by $G_X$
and write $G_U$ for the original $G$.
\par\smallskip
(iii) The principal $G$-bundle $P$ is defined over $\Ysc\times_{\Spec \mathcal O}X$ and the function $\text{f}\in\mathcal O$
belongs to $k[X]$.
\par\smallskip
(iv) The restriction $P_{\text{f}}$ of the bundle $P$ to the open subset $\Ysc\times_{\Spec\mathcal O}
X_{\text{f}}$ is trivial and $\text{f}$ vanishes at each $x_i$'s.
\par\smallskip
In particular, now we are given the smooth irreducible affine
$k$-scheme $X$, the finite family of points $x_1,x_2,\dots,x_n$ on
$X$, and the non-zero function $\ttf\in k[X]$ vanishing at each
point $x_i$. It was shown in~\cite[Section 5]{PaSV} that, starting with these data,
one can constract what is called there {\it a nice triple}~\cite[Def. 4.1]{PaSV}, of the form $(q_U:\mathcal X\to U,f,\Delta)$.
This triple fits into a commutative diagram
\begin{equation}
\label{SquareDiagram}
    \xymatrix{
     \mathcal X\ar[d]_{q_U}\ar[rr]^{q_X} &&  X   & \\
     U \ar[urr]_{\can} \ar@/_0.8pc/[u]_{\Delta} &\\
    }
\end{equation}
where
\begin{equation}
\label{DeltaQx} q_X\circ\Delta=\can
\end{equation}
and
\begin{equation}
\label{DeltaQu} q_U\circ\Delta=\id_U.
\end{equation}
Moreover, $f:=q_X^*(\ttf)$.
We did that shrinking $X$ along the way, but all properties (i) to (iv) were preserved.

In particular, the restriction $P_{\ttf}$ of
the bundle $P$ to the open subscheme $\Ysc\times_{\Spec \mathcal O} X_{\tt f}$ is
trivial by Item (iv) above.

Set $G_{\mathcal X}:=(q_X)^*(G)$, and let $G_{\const}$ be the pull-back of $G_U$ to
$\mathcal X$ via $q_U$. By~\cite[Theorem 4.3]{PaSV}
there exists a {\it morphism of nice triples}~\cite[Def. 4.2]{PaSV}
$$ \theta: (\mathcal X^{\prime},f^{\prime},\Delta^{\prime})\to (\mathcal X,f,\Delta) $$
\noindent
and an isomorphism
\begin{equation}
\label{KeyEquation} \Phi: \theta^*(G_{\const}) \to
\theta^*(G_{\mathcal X})=: G_{\mathcal X^{\prime}}
\end{equation}
of $\mathcal X^{\prime}$-group schemes such that
$(\Delta^{\prime})^*(\Phi)=\id_{G_U}$.
\par
Set
\begin{equation}
\label{QprimeX} q^{\prime}_X=q_X\circ\theta:\mathcal X^{\prime}\to X.
\end{equation}
Recall that
\begin{equation}
\label{QprimeU} q^{\prime}_U=q_U\circ\theta:\mathcal X^{\prime}\to
U,
\end{equation}
since
$\theta$
is a morphism of nice triples.

Consider the pullback $(q^{\prime}_X)^*(P)$ of $P$ from $\Ysc\times_U X$ to $\Ysc\times_U \mathcal X^{\prime}$ as a principal
$(q^{\prime}_U)^*(G_U)=\theta^*(G_{\const})$-bundle via the isomorphism
$\Phi$.

Recall that $P$ is trivial as a $G$-bundle over
$\Ysc\times_U X_{\tt f}$. Therefore, $(q^{\prime}_X)^*(P)$ is trivial as a
principal $G_{\mathcal X^{\prime}}$-bundle over $\Ysc\times_U\mathcal X^{\prime}_{\theta^*(f)}$.
Since $\theta$ is a nice triple morphism one has
$f^{\prime}=\theta^*(f)\cdot g^{\prime}$, and thus the principal
$G_{\mathcal X^{\prime}}$-bundle
$(q^{\prime}_X)^*(P)$ is trivial over
$\Ysc\times_U\mathcal X^{\prime}_{f^{\prime}}$.
\par
We conclude that $(q^{\prime}_X)^*(P)$ is trivial over
$\Ysc\times_U\mathcal X^{\prime}_{f^{\prime}}$, when regarded as a principal
$G_U$-bundle (more precisely, $(q^{\prime}_U)^*(G_U)$-bundle; we omit this base change from the notation) via the isomorphism $\Phi$.
\par
By~\cite[Theorem 4.5]{PaSV} there exists a finite
surjective morphism $\sigma:\mathcal X^{\prime}\to\Aff^1\times U$
of $U$-schemes satisfying
\begin{itemize}
\item[\rm{(1)}] $\sigma$ is \'{e}tale along the closed subset
$\{f'=0\}\cup\Delta'(U)$.
\item[\rm{(2)}] For a certain element $g_{f',\sigma}\in\Gamma(\mathcal X^{\prime},\mathcal O_{\mathcal X^{\prime}})$ and
a unitary polynomial $N(f')\in\mathcal O[t]\hra\Gamma(\mathcal X^{\prime},\mathcal O_{\mathcal X^{\prime}}) $,  defined by the
distinguished $\sigma$ as in~\cite[Section 4]{PaSV}, one has
$$ \sigma^{-1}\Big(\sigma\big(\{f'=0\}\big)\Big)=\{N(f')=0\}=\{f'=0\}\sqcup\{g_{f',\sigma}=0\}. $$

\item[\rm{(3)}] Denote by $(\mathcal X')^0\hra\mathcal X'$ the largest
open sub-scheme, where the morphism $\sigma$ is \'{e}tale. Write
$g'$ for $g_{f',\sigma}$ from now on. Then the square
\begin{equation}
\label{ElemNisSquareDiagram}
    \xymatrix{
(\mathcal X^{\prime})^{0}_{N(f^{\prime})}= (\mathcal
X^{\prime})^{0}_{f^{\prime}g^{\prime}}\ar[rr]^{\inc}
\ar[d]_{\sigma^0_{f^{\prime}g^{\prime}}} &&
(\mathcal X^{\prime})^0_{g^{\prime}}\ar[d]^{\sigma^0_{g^{\prime}}}  &\\
(\Aff^1\times U)_{N(f^{\prime})}\ar[rr]^{\text{\rm inc}}&&\Aff^1\times U &\\
}
\end{equation}
is an elementary Nisnevich square. (Here $\sigma^0_{f^{\prime}g^{\prime}}$ and $\sigma^0_{g^{\prime}}$ stand for the
corresponding restrictions of $\sigma$.)

\item[\rm{(4)}] One has $\Delta'(U)\subset(\mathcal X')^{0}_{g'}$.
\end{itemize}

Regarded as a principal $G_U$-bundle via the isomorphism $\Phi$,
the bundle $(q^{\prime}_X)^*(P)$ over $\Ysc\times_U\mathcal X^{\prime}$
becomes trivial over $\Ysc\times_U\mathcal X^{\prime}_{f^{\prime}}$, and a
fortiori over $\Ysc\times_U(\mathcal X^{\prime})^{0}_{f^{\prime}g^{\prime}}$.
Now, gluing the trivial $G_U$-bundle over $\Ysc\times_U(\Aff^1\times U)_{N(f^{\prime})}$ to the bundle $(q^{\prime}_X)^*(P)$ along the isomorphism
\begin{equation}
\label{Skleika}
\psi: \Ysc\times_U(\mathcal X^{\prime})^{0}_{N(f^{\prime})} \times_U G_U \to
(q^{\prime}_X)^*(P)|_{\Ysc\times_U(\mathcal X^{\prime})^{0}_{N(f^{\prime})}}
\end{equation}
of principal $G_U$-bundles, we get a principal $G_U$-bundle $P_t$
over $\Ysc\times_U(\Aff^1\times U)$ such that
\begin{itemize}
\item[1.] it is trivial over $\Ysc\times_U(\Aff^1\times U)_{N(f^{\prime})}$,
\item[2.] $(\sigma)^*(P_t)$ and $(q^{\prime}_X)^*(P)$ are
isomorphic as principal $G_U$-bundles over $\Ysc\times_U (\mathcal X^{\prime})^0_{g^{\prime}}$. Here $(q^{\prime}_X)^*(P)$
is regarded as a principal $G_U$-bundle via the $\mathcal
X^{\prime}$-group scheme isomorphism $\Phi$ from
(\ref{KeyEquation}); 
\item[3.] over $\Ysc\times_U(\mathcal X^{\prime})^{0}_{N(f^{\prime})}$ the
two $G_U$-bundles are identified via the isomorphism $\psi$ from
(\ref{Skleika}).
\end{itemize}

Finally, form the following diagram
\begin{equation}
\label{DeformationDiagram2}
    \xymatrix{
\Aff^1_U\ar[drr]_{\pr}&&(\mathcal
X^{\prime})^0_{g^{\prime}}\ar[d]^{}
\ar[ll]_{\sigma^0_{g^{\prime}}=\sigma|_{(\mathcal X^{\prime})^0_{g^{\prime}}}}\ar[d]_{q^{\prime}_U}
\ar[rr]^{q^{\prime}_X}&&X  &\\
&&U\ar[urr]_{\can}\ar@/_0.8pc/[u]_{\Delta^{\prime}} &\\
    }
\end{equation}
This diagram is well-defined since by Item (4) above the image of the morphism
$\Delta^{\prime}$ lands in $(X^{\prime})^0_{g^{\prime}}$.

\begin{lem}
\label{lem:MainData} The unitary polynomial $h=N(f^{\prime})\in \mathcal O[t]$, the
principal $G_U$-bundle $P_t$ over $\Ysc\times_U\Aff^1_U$, the diagram
{\rm(\ref{DeformationDiagram2})} and the isomorphism
{\rm(\ref{KeyEquation})} constructed above
has the following properties:
\par\smallskip
(1*) $q^{\prime}_U=\pr\circ\sigma$,
\par\smallskip
(2*) $\sigma$ is \'etale,
\par\smallskip
(3*) $q^{\prime}_U\circ\Delta^{\prime}=\id_U$,
\par\smallskip
(4*) $q^{\prime}_X\circ\Delta^{\prime}=\can$,
\par\smallskip
(5*) the restriction of $P_t$ to $Y\times_U(\Aff^1_U)_h$ is a trivial
$G_U$-bundle,
\par\smallskip
(6*) $(\sigma)^*(P_t)$ and $(q^{\prime}_X)^*(P)$ are isomorphic as principal
$G_U$-bundles over $\Ysc\times_U (\mathcal X^{\prime})^0_{g^{\prime}}$. Here $(q^{\prime}_X)^*(P)$ is regarded as a principal
$G_U$-bundle via the group scheme isomorphism $\Phi$.

\end{lem}

\begin{proof} By the very choice of $\sigma$ it is an $U$-scheme
morphism, which proves (1*). Since $(\mathcal
X^{\prime})^0\hra\mathcal X^{\prime}$ is the largest open
sub-scheme where the morphism $\sigma$ is \'etale, one gets (2*).
Property (3*) holds for $\Delta^{\prime}$ since $(\mathcal
X^{\prime},f',\Delta^{\prime})$ is a nice triple
and, in particular, $\Delta^{\prime}$ is a section of
$q^{\prime}_U$. Property (4*) can be established as follows:
$$ q^{\prime}_X\circ\Delta^{\prime}=
(q_X\circ\theta)\circ\Delta^{\prime}=q_X\circ\Delta=\can. $$
\noindent
The first equality here holds by the definition of $q^{\prime}_X$,
see (\ref{QprimeX}); the second one holds, since $\theta$ is a
morphism of nice triples; the third one follows from
(\ref{DeltaQx}). Property (5*) is just Property 1 in the above
construction of $P_t$.
Property (6*) is precisely Property 2 in our construction of
$P_t$.
\end{proof}

One readily sees that the properties in Lemma~\ref{lem:MainData} imply that {\it if the $G$-bundle $P_t$
is trivial on $\Ysc\times_U \Aff^1_U$, then the original bundle $P$ is trivial on $\Ysc$.}

Indeed, if $P_t$ is trivial, then by Property (6*) in Lemma~\ref{lem:MainData} the $G_U$-bundle
$(q^{\prime}_X)^*(P)$  over $\Ysc\times_U (\mathcal X^{\prime})^0_{g^{\prime}}$ is trivial as well.
Hence, using Property (4*), we deduce that the bundle $(\Delta^{\prime})^*\bigl((q^{\prime}_X)^*(P)\bigr)=\can^*(P)$ is a trivial
$(\Delta^{\prime})^*\bigl((q^{\prime}_X)^*(G)\bigr)=\can^*(G)$-bundle over $Y\times_U U=Y$.

\section{Proof of Theorems~\ref{th:main} and~\ref{th:main-geometric}}\label{sec:mainproof}

The following easy lemma was essentially proved inside the proof of~\cite[Theorem 8.6]{PaSV}. Here we provide
a more detailed proof in a slightly more general situation.

\begin{lem}\label{lem:linear}
Let $R$ be a semilocal ring, $G$ a simply connected semisimple group scheme over $R$. There exists
a closed embedding $G\to\GL_{n,R}$ for some $n\ge 1$.
\end{lem}
\begin{proof}
We can assume without loss of generality that $R$ is connected.
Let $U=\Spec R$. The $U$-group scheme $G$ is given by a $1$-cocycle
$\xi \in Z^1(U, Aut(G_0))$, where $G_0$ is the split simply connected simple group scheme over $U$
of the same type as $G$, and $Aut(G_0)$ is the automorphism group scheme of $G_0$.
Recall that
$Aut(G_0) \cong G^{\ad}_0 \rtimes N$,
where $N$ is the finite group of automorphisms of the Dynkin diagram of $G_0$,
and $G^{\ad}_0$
is the adjoint group corresponding to $G_0$.
Since $Aut(G_0) \cong G^{\ad}_0 \rtimes N$,
we have an exact sequence of pointed sets
$$\{1\} \to H^1(U, G^{\ad}_0) \to H^1(U, G^{\ad}_0 \rtimes N) \to H^1(U,N).$$
Thus there is a finite \'{e}tale morphism
$\pi: V \to U$ such that
$G_V:=G \times_U V$
is given by a
$1$-cocycle $\xi_V \in Z^1(U, G^{\ad}_0)$. We can choose $V$ so that $V/U$ is moreover a Galois extension.


For each fundamental weight $\lambda$ of $G_0$, there is a central (also called center preserving, see~\cite{PS-tind})
representation
$\rho_{\lambda}: G_0 \to GL_{V_{\lambda}\otimes_{\ZZ} U}$,
where $V_{\lambda}$ is the Weyl module over $\ZZ$ corresponding to $\lambda$.
This gives a commutative diagram of $U$-group morphisms
\begin{equation}
\label{Representation1}
    \xymatrix{
G_0 \ar[rr]^{\rho_{\lambda}} \ar[d]_{} && GL_{V_{\lambda}\otimes_{\ZZ} U} \ar[d]^{}& \\
G^{\ad}_0 \ar[rr]^{\bar \rho_{\lambda}} && \PGL_{V_{\lambda}\otimes_{\ZZ} U}.      & \\  }
\end{equation}
Considering the product of $\rho_{\lambda}$'s with $\lambda$ running over the set $\Lambda$ of all fundamental
weights, we obtain the following commutative diagram
of algebraic $k$-group homomorphisms:
\begin{equation}
\label{Representation2}
    \xymatrix{
G_0 \ar[rr]^{\rho} \ar[d]_{} && \prod_{\lambda \in \Lambda}GL_{V_{\lambda}\otimes_{\ZZ} U} \ar[d]^{}& \\
G^{\ad}_0 \ar[rr]^{\bar \rho} && \prod_{\lambda \in \Lambda}\PGL_{V_{\lambda}\otimes_{\ZZ} U}.      & \\  }
\end{equation}
By the definition of Weyl modules, $\rho$ is a closed embedding (cf.~\cite[Lemma 2]{PS-tind}).

Twisting the $V$-group morphism $\rho$ with the $1$-cocycle $\xi_V$ we get
an $V$-group scheme morphism
$\rho_V: G_V \to \prod_{\lambda \in \Lambda} GL_1(A_{\lambda})$,
where the product is a product of group schemes over $V$, and each
$A_{\lambda}$
is an Azumaya algebra over $V$ obtained from
$\End(V_{\lambda}\otimes_{\ZZ} U)$
via the $1$-cocycle
$\theta_{\lambda}=(\bar \rho_\lambda)_*(\xi_V) \in Z^1(V, \PGL_{V_{\lambda}\otimes_{\ZZ} U})$.
Composing $\rho_V$ with the natural closed embedding $\prod_{\lambda \in \Lambda} GL_1(A_{\lambda})\hra GL_{\bigoplus A_\lambda}$,
we obtain a closed embedding
$$
G_V\hra \GL_{m,V},
$$
for a large enough integer $m$.


One has
$$Hom_V(G_V,\GL_{m,V})=Hom_U(G,R_{V/U}(\GL_{m,V})),$$
where $R_{V/U}$ is the Weil restriction functor. Thus $\rho_V$ determines
an $U$-morphism
$$\rho_U: G \hra R_{V/U}(\GL_{m,V}).$$
Here $\rho_U$ is a $U$-group scheme morphism, and, since $\rho$ is a closed embedding,
$\rho_U$ is a closed embedding as well (\'{e}tale descent).

Let $d$ be the degree of the Galois extension $V=\Spec S$ over $U=\Spec R$.
The $U$-group scheme $R_{V/U}(\GL_{m,V})$ admits a natural closed embedding into $\GL_{md,U}$,
such that, for any $R$-algebra $X$, the image of $g\in R_{V/U}(\GL_{m,V})(X)=\GL_m(X\otimes_R S)$
is the corresponding element of $\GL_{md}(X)$, the $X$-linear automorphism of $X^{\oplus md}\cong (X\otimes_R S)^{\oplus m}$.
Now, composing this embedding with $\rho_U$, we obtain a closed embedding
$G\hra R_{V/U}(\GL_{m,V})\hra \GL_{n,U}$, for $n=md$.

\end{proof}

\begin{thm}\label{thm:local}
Let $B$ be a semi-local Noetherian ring containing an infinite field. Let $G$ be an isotropic simply connected simple
group scheme over $B$.
Let $P$ be a principal $G$-bundle over $\Aff^1_B$ trivial over $(\Aff^1_B)_f$ for a monic polynomial $f\in B[t]$. Then
$P$ is trivial.
\end{thm}
\begin{proof}
This theorem was proved in~\cite{PaSV}. Indeed, this is precisely~\cite[Theorem 2.1]{PaSV}, except that
in that theorem the base ring $B$ was required to be ``of geometric type'', i.e. a semilocal ring of finitely many points on a smooth
variety over an infinite field. However, tracing the proof of this statement,
one readily sees that the only properties of $B$ that are used are that $B$ is semi-local, Noetherian, and
contains an infinite field. (The ``geometric type'' assumption was an umbrella assumption in the most part
of~\cite{PaSV}, since it is crucial for the validity of the main theorem
~\cite[Theorem 1.2]{PaSV}.)

\end{proof}

\begin{proof}[Proof of Theorem~\ref{th:main-geometric}]
Consider the case where $R$ is a semi-local ring of several points on a $k$-smooth scheme over an infinite field $k$ (the ``geometric case'').
Let $P$ be a principal $G$-bundle which is in the kernel of the map
$$
H^1_{\text{\'et}}(\Asc\times_{\Spec k}\Spec R,G)\to H^1_{\text{\'et}}(\Asc\times_{\Spec k} \Spec K,G).
$$
By considerations in~\S~\ref{sec:constr} there is a principal $G$-bundle $P_t$ over
$\Asc\times_k\Aff^1_{R}=\Aff^1_{A\otimes_k R}$, trivial
over $\Asc\times_k(\Aff^1_{R})_f$ for a monic polynomial $f\in R[t]$, and such that if $P_t$ is trivial on the whole
$\Asc\times_k\Aff^1_{R}$,
then the original $G$-bundle $P$ over $\Asc\times_k\Spec R$ is trivial as well. Thus, it is enough to
show that $P_t$ is trivial.

Since $R$ is a semilocal ring containing an infinite field, and $f$ is monic, the Chinese remainder theorem
implies that there is $a\in R$ with $f(a)\in R^\times$; changing the variable, we can assume that $f(0)\in R^\times$.

Set $B=A\otimes_k R$. Note that $B$ is a Noetherian commutative ring containing an infinite field $k$.
By Theorem~\ref{thm:local} for any localization $B_m$ of $B$ at a maximal ideal $m\subseteq B$, the bundle $P_t\times_{\Spec B}\Spec B_m$
is trivial. By Lemma~\ref{lem:linear} the group scheme $G$ admits a closed embedding into some $\GL_n$ over $R$,
and hence, by base change, over $B$. Thus, we are given a principal $G$-bundle $P_t$ over $\Aff^1_B=\Aff^1_k\times_k\Spec B$, which is
trivial Zariski-locally in $\Spec B$, as well as on $\{0\}\times \Spec B$; and $G$ is a linear group. Then
by Moser's local-global principle~\cite[Korollar 3.5.2]{Mo} $P_t$ is trivial on $\Aff^1_B=\Asc\times_k\Aff^1_{R}$.

\end{proof}

\begin{proof}[Proof of Theorem~\ref{th:main}]
The claim follows from Theorem~\ref{th:main-geometric} via the well-known result of D. Popescu~\cite{Po,Swan}.
Since the field $k$ is perfect, the morphism $k\to R$ is geometrically regular. Therefore, by Popescu's theorem
$R$ is a filtered direct limit of smooth $k$-algebras. One readily sees that, since $R$ is semilocal,
these smooth $k$-algebras can also be chosen to be semilocal rings of several points on a smooth k-variety.
Since the functor $H^1_{\et}(-,G)$ commutes with filtered direct limits, the result follows.

\end{proof}

\end{document}